\documentclass[12pt]{amsart}
\usepackage{amssymb}
\usepackage{fullpage}
\usepackage{amsrefs}

\theoremstyle{definition}
\newtheorem{definition}{Definition}

\theoremstyle{theorem}
\newtheorem{theorem}{Theorem}
\newtheorem{lemma}{Lemma}
\newtheorem{corollary}{Corollary}

\newcommand{\C}{\mathbb{C}}
\newcommand{\D}{\mathbb{D}}

\title{A simple proof of necessity in the McCullough-Quiggin theorem}
\author{Greg Knese}
\date{\today}                                           
\address{Washington University in St. Louis \\ Department of Mathematics \& Statistics\\
One Brookings Drive \\ St. Louis, MO 63130}
\email{geknese@wustl.edu}
\thanks{Partially supported by NSF grant DMS-1900816}

\begin{document}
\maketitle

\begin{abstract}
A short and simple proof of necessity in the McCullough-Quiggin
characterization of positive semi-definite kernels
with the complete Pick property is presented.
\end{abstract}

\section{Introduction}

Given $N$ points $z_1,\dots,z_N$ in the unit
disk $\D \subset\C$ and $N$ ``targets'' $w_1,\dots, w_N\in \C$, 
when does there exist a holomorphic function 
$f:\D\to \overline{\D}$ that interpolates: $f(z_i)=w_i$ for
$i=1,\dots, N$?  Pick's theorem of 1916 says this
interpolation can be done if and only if the matrix
\[
\left( \frac{1-w_i \bar{w}_j}{1-z_i\bar{z}_j}\right)_{i,j} \text{ is positive semi-definite \cite{Pick}.}
\]
In the hundred intervening years, this theorem
has been generalized and reinterpreted a number of ways;
see \cite{McC}.
The theorem generalizes to matrix valued functions, and
it can be reinterpreted as a special property
of multipliers of reproducing kernel Hilbert spaces.  
The remarkable McCullough-Quiggin theorem
precisely describes which reproducing kernel
Hilbert spaces have this property.
It is not our goal to delve into
motivation and background, as this
is already done in the paper \cite{AM2000}
and related book \cite{AMbook}.
Instead, our goal is to give a short
and simple
proof of necessity in the McCullough-Quiggin theorem.

Let $k:X\times X \to \C$ be a positive semi-definite (PSD)
kernel on $X$ with associated 
reproducing kernel Hilbert space $H$. 
Let $L_1,L_2$ be auxiliary Hilbert spaces and 
let $B(L_1,L_2)$ denote the bounded
linear operators from $L_1$ to $L_2$.
A function $\Phi: X\to B(L_1,L_2)$ is a multiplier
of norm at most one, i.e.\ belongs to 
$\text{Mult}_1(H\otimes L_1, H\otimes L_2)$,
if and only if for every $f\in H\otimes L_1$
we have $\Phi f \in H\otimes L_2$ and $\|\Phi f\| \leq \|f\|$.
This property is equivalent to the property that the operator valued kernel
\[
(I-\Phi(x)\Phi(y)^*) k(x,y) \text{ is PSD}.
\]

\begin{definition}
A kernel $k$ has the \emph{complete Pick property} if for
all natural numbers $s,t$, 
whenever $S\subset X$ is finite
and $W\in \text{Mult}_1(H|_S\otimes \C^t, H|_S\otimes \C^s)$
then there exists 
$\Phi \in \text{Mult}_1(H\otimes \C^t, H\otimes \C^s)$
such that $\Phi|_S = W$.
\end{definition}


\begin{definition}
The kernel $k$ is irreducible if $k(x,y) \ne 0$
for all $x,y \in X$ and $k_x,k_y$ are linearly
independent for $x\ne y$.
\end{definition}

\begin{theorem}[McCullough \cites{McCullough1, McCullough2}, Quiggin \cite{Quiggin1}] \label{mainthm}
Let $k:X\times X\to \C$ be PSD and irreducible.  
Then, $k$ has the complete Pick property if and only if
for each $z\in X$
\[
F_z(x,y) := 1 - \frac{k(x,z) k(z,y)}{k(z,z) k(x,y)} 
\text{ is PSD.}
\]
\end{theorem}

Agler-McCarthy \cite{AM2000}
first formulated and proved the theorem in precisely 
the form above.  
The original proofs
dig into precisely what needs to be satisfied
in order to extend a multiplier on
a set of points to one more point
(``the one point extension property'').
Some sort of axiom of choice is 
invoked to build a multiplier on all 
of $X$.
The proof of necessity given below is a short
inductive proof.  
One innovation worth mentioning is that we do not use 
full irreducibility of $k$; we only need
$k$ to be non-vanishing. 
A straightforward proof of sufficiency
is due to Ball-Trent-Vinnikov \cite{BTV}.
It proves a realization formula for multipliers
that builds a multiplier directly and uses
Hilbert space geometry instead of the axiom of 
choice. 
The master's thesis \cite{Marx} contains
a nice treatment.

While this paper is about a foundational aspect
of complete Pick kernels, a lot of fascinating work has been
done in recent years on more advanced aspects
of these kernels and related Drury-Arveson type spaces.
See \cites{Aleman, CH, JM} for a sampling of
 some recent literature. In particular, the
 paper \cite{CH} contains an interesting
 characterization of the complete Pick
 property in terms of completely
 contractive embeddings.

We shall assume knowledge of the rudiments
of vector-valued reproducing kernel Hilbert spaces 
including multipliers and the
Schur product theorem. See \cites{AMbook, PR}.

\section*{Acknowledgements}

The proof below was found while
teaching a topics course on reproducing
kernel Hilbert spaces at Washington University
in the Spring of 2019.
I thank the students in that course for
letting me test the proof out on them.
I thank Michael Hartz, John McCarthy, 
and Scott McCullough for helpful discussions.



\section{Proof of Necessity}
We need three basic lemmas which apply 
to any PSD kernel $k$ with associated reproducing kernel
Hilbert space $H$ such that $k(z,z) \ne 0$ for
some fixed $z\in X$.  Set 
\[
k^z(x,y) := k(x,y) - \frac{k(x,z)k(z,y)}{k(z,z)}.
\]

\begin{lemma} \label{formula}
Consider the closed subspace 
$H_z = \{f\in H: f(z) = 0\}.$
The reproducing kernel for $H_z$ is $k^z(x,y)$.
\end{lemma}
\begin{proof}
Note $H = H_z \oplus \C k_z$.  
The reproducing kernel for $\C k_z$ is
simply $k(z,z)^{-1} k_z(x) \overline{k_z(y)}$.  
Then, $k^z_y(x) = k^z(x,y)$ as defined above
belongs to $H_z$ and reproduces elements of $H_z$.
\end{proof}

\begin{lemma}
Let $L$ be a Hilbert space.
Given $f \in H \otimes L$, we have $f(z) = \vec{0}$
if and only if $f \in H_z\otimes L$.  
\end{lemma}
\begin{proof}
Evidently, $f\in H_z\otimes L$ implies $f(z)=\vec{0}$
because the tensor project is a closed
span of elements with this property.
Conversely, let $P$ be the orthogonal projection from $H\otimes L$
to $H_z\otimes L$.
Let $f\in H\otimes L$ with $f(z)=\vec{0}$.
Note by Lemma \ref{formula}
\[
\langle f, k^z_y\otimes v\rangle = 
\langle f, (k_y - k_z \frac{k(z,y)}{k(z,z)})\otimes v\rangle=
\langle f(y), v\rangle_L.
\]
On the other hand, since
$k^z_y \otimes v = P(k_y^z\otimes v) \in H_z\otimes L$, 
$\langle f, k^z_y\otimes v\rangle = 
\langle Pf, k_y^z\otimes v\rangle =
\langle Pf(y), v\rangle_L.$ 
So, $f = Pf \in H_z\otimes L$.\end{proof}

\begin{lemma}
Let $L_1,L_2$ be Hilbert spaces.
If $\Phi \in \text{Mult}_1(H\otimes L_1, H\otimes L_2)$, 
then
\[
\Phi \in \text{Mult}_1(H_z\otimes L_1, H_z\otimes L_2).
\]
\end{lemma}
\begin{proof}
If $f\in H_z\otimes L_1$ then $\Phi f \in H\otimes L_2$ and 
$\Phi(z)f(z) = \vec{0}$ so that $\Phi f \in H_z\otimes L_2$.
The inequality $\|\Phi f\| \leq \|f\|$ holds for $f\in H_z\otimes L_1\subset H\otimes L$
so $\Phi \in \text{Mult}_1(H_z\otimes L_1,H_z\otimes L_2)$.
\end{proof}

The following corollary is the PSD kernel interpretation of the
above lemma.

\begin{corollary} \label{multcor}
If $\Phi: X \to B(L_1, L_2)$ is a function such that
$(I-\Phi(x) \Phi(y)^*)k(x,y)$ is PSD then $(I-\Phi(x)\Phi(y)^*) k^z(x,y)$
is PSD.
\end{corollary}

We now assume $k$ is non-vanishing and that
the complete Pick property holds for $k$.
We proceed to show $F_z$ from the statement of Theorem \ref{mainthm} is PSD. 
For $x,z\in X$, $F_z(x,x) = 1 - \frac{|k(x,z)|^2}{k(x,x)k(z,z)}\geq 0$ by Cauchy-Schwarz.  
Now take any $x_1,\dots, x_{N}, x_{N+1} \in X$ where $N \geq 2$.
Assuming $( F_{x_N}(x_i,x_j))_{i,j=1,\dots, N-1} \geq 0$
we will show
$( F_{x_{N+1}}(x_i,x_j))_{i,j=1,\dots, N} \geq 0$.
Note that since $F_{x_N}(x_i,x_N)=0$, 
the matrix expanded with zeros 
$A := ( F_{x_N}(x_i,x_j))_{i,j=1,\dots, N}$ is PSD.
Factor the entries of $A$ as $A_{i,j} = v_i v_j^*$ using row vectors $v_i$.
For simplicity we will write $k_{ij} = k(x_i,x_j)$ below.
Note that
\[
(1-A_{i,j}) = (1-v_i v_j^*) = \frac{k_{iN} k_{Nj}}{k_{NN} k_{ij}}
\]
so that
$(1-v_iv_j^*)k_{ij} = \frac{k_{iN} k_{Nj}}{k_{NN}}$
is rank one and PSD.
By the complete Pick property, there exists a contractive multiplier $\Phi$
with $\Phi(x_i)= v_i$.
By Corollary \ref{multcor} applied to $z=x_{N+1}$,
\[
(1-v_iv_j^*) \left(k_{i,j} - \frac{k_{i,N+1} k_{N+1,j}}{k_{N+1, N+1}}\right) 
=
\frac{k_{iN} k_{Nj}}{k_{NN}} F_{x_{N+1}}(x_i,x_j) \text{ is PSD.}
\]
Here $i,j=1,\dots, N$.
The matrix $\left(\frac{k_{NN}}{k_{iN} k_{Nj}}\right)_{i,j}$
is rank one and PSD so by the Schur product theorem
we see that $(F_{x_{N+1}}(x_i,x_j))_{i,j=1,\dots N}$ is PSD.
By induction, this proves that $F_z(x,y)$ is PSD for all $z\in X$.

\end{document}